\documentclass[12pt,english]{article}
\usepackage{amsmath, amsthm, latexsym, amssymb,url}
\usepackage{amsfonts}
\usepackage{xypic}
\usepackage{epsfig}

\usepackage{graphicx}% http://ctan.org/pkg/graphicx
\usepackage{yhmath}% http://ctan.org/pkg/yhmath
\usepackage{mathdots}% http://ctan.org/pkg/mathdots
\usepackage{mathtools} 
\usepackage{pifont}
\usepackage{diagbox}
\usepackage{makecell}

\input xy
\xyoption{all}

\newcommand{\shrinkmargins}[1]{
  \addtolength{\textheight}{#1\topmargin}
  \addtolength{\textheight}{#1\topmargin}
  \addtolength{\textwidth}{#1\oddsidemargin}
  \addtolength{\textwidth}{#1\evensidemargin}
  \addtolength{\topmargin}{-#1\topmargin}
  \addtolength{\oddsidemargin}{-#1\oddsidemargin}
 \addtolength{\evensidemargin}{-#1\evensidemargin}
  }

\shrinkmargins{.7}

\theoremstyle{plain}

\newtheorem{theorem}{Theorem}[section]
\newtheorem{corollary}[theorem]{Corollary}
\newtheorem{lemma}[theorem]{Lemma}
\newtheorem{proposition}[theorem]{Proposition}
\newtheorem{question}[theorem]{Question}

\newtheorem*{teo}{Theorem}
\newtheorem*{defi}{Definition}

\newtheorem{definition}[theorem]{Definition}

\theoremstyle{remark}
\newtheorem{remark}[theorem]{Remark}

\theoremstyle{definition}
\newtheorem{example}[theorem]{Example}

\theoremstyle{fact}

\theoremstyle{claim}

\def \Z { \mathbb{Z}}
\def \Q { \mathbb{Q}}

\def \C { \mathbb{C}}

\def \det { \text{det}}

\def \tr { \text{tr}}

\begin{document}

\thispagestyle{empty}
\setcounter{tocdepth}{7}

\title{Weak arithmetic equivalence.}
\author{Guillermo Mantilla-Soler}

\date{}

\maketitle

\maketitle

\begin{abstract}
Inspired by the invariant of a  number field given by its zeta function, we define the notion of {\it weak arithmetic equivalence}, and show that under certain ramification hypotheses, this equivalence determines the local root numbers of the number field. This is analogous to a result of Rohrlich  on the local root numbers of a rational elliptic curve.  Additionally, we prove that for tame non-totally real number fields, the  integral trace form is invariant under arithmetic equivalence.\\
\end{abstract}

\section{ Introduction}
 
One of the most fundamental arithmetical invariants of a number field is its Dedekind zeta function. It is well known that pairs of number fields with the same zeta function, {\it Arithmetically equivalent number fields}, share many arithmetic invariants. Among them the discriminant, unit group, signature, the product of class number times regulator and some others (see \cite[III, \S1]{Klingen}). A classic result of R. Perlis (see \cite[Corollary to Theorem 1]{Perlis}) states that any two arithmetically equivalent number fields have isometric rational trace forms. Since the rational trace form per se is not of an arithmetic nature, we are interested to see how arithmetic equivalence relates to the arithmetic version of the trace form i.e., the integral trace form. We have studied this briefly in the past (see \cite[\S2]{Manti1}) and have seen that in order to obtain any implication along the lines of  Perlis, we must avoid number fields with wild ramification. In this paper, we study this relation in detail. For example we show that any two tame arithmetically equivalent number fields that are ramified at infinity have always isometric integral trace forms. Furthermore, we define a ``finite" version of arithmetic equivalence and we show that under restricted conditions, such an invariant determines the integral trace form. We also exhibit a relation between our new defined invariant and the local root numbers associated to the number field in question (see Theorem \ref{LAEIsoLocRoot}). This last point of view is analogous to a result of Rohrlich that shows that the bad part of the $L$-function of a semistable elliptic curve determines its local root numbers (see Theorem \ref{rohrlic}).  
 
 \subsection{Motivation and results}
To understand the motivation behind our definitions, it's important to review Perlis' result on the rational trace form. Let $K$ be a number field of degree $n$, and let $s_{n}$ be the symmetric group in $n$ symbols. The Dedekind zeta function $\zeta_{K}$ is the  $L$-function of an Artin representation $\rho_{K}$ of $G_{\Q}$, namely the representation obtained by composing the permutation representation $\pi_{K}: G_{\Q} \to s_{n}$ and the natural inclusion $j: s_{n} \to {\rm GL}_{n}(\C)$. Since we are interested in the equivalence class of the representation $\rho_{K}$, we think of it as an element in ${\rm H}^{1}(\Q, {\rm GL}_{n}(\C))$, and the same for $\pi_{K}$. The natural inclusion \[\iota:  s_{n} \to O_{n}(\overline{\Q})\] induces a map of pointed sets \[\iota^{*}:  {\rm H}^{1}(\Q, s_{n}) \to {\rm H}^{1}(\Q, O_{n}).\] Since ${\rm H}^{1}(\Q, O_{n})$ classifies isometry classes of non-degenerate rational quadratic forms of dimension $n$, there exists a quadratic form corresponding to $\iota^{*}(\pi_{K})$.  Perlis' realization(see \cite[Lemma 1.b]{Perlis}) is that such a form is precisely {\it the rational trace form} i.e., the rational quadratic form associated to the bilinear pairing 
\[
\begin{array}{cccc}
& K \times K  &  \rightarrow & \Q  \\  & (x,y) & \mapsto &
\tr_{K/\Q}(xy).
\end{array}\]

The above result can be interpreted as a relation between the rational trace form of the field $K$ and the representation $\rho_{K}$. Presumably such a relation led Perlis to:

\begin{theorem}[Perlis]
Let $K$ and $L$ be arithmetically equivalent number fields. Then, $K$ and $L$ have isometric rational trace forms.
\end{theorem}

The main ideas behind Perlis' proof of the above are the following: Using formulas of Serre for the local Hasse invariants of the trace form Perlis shows that for every prime $p$ the local $p$-Hasse invariant of the trace form $\iota^{*}(\pi_{K})$ can be written in terms of the $p$-local Stiefel-Whitney class of the representation $\rho_{K}$. Moreover, due to a formula of Deligne, such numbers can be written in terms of local root numbers of the representations $\rho_{K}$ and $\det(\rho_{K})$ (see \S\ref{Dedekind} for details). By evaluating $\rho_{K}$ at complex conjugation, it can be seen that the signature of $\iota^{*}(\pi_{K})$ is determined by the representation $\rho_{K}$. It follows, thanks to the Hasse principle, that the isometry class of the rational trace $\iota^{*}(\pi_{K})$ is completely determined by the representation $\rho_{K}$. On the other hand, by the Chebotarev's density theorem, the representation $\rho_{K}$ is completely determined by  $\zeta_{K}$. In particular, two number fields with the same Dedekind zeta function share their rational trace.\\

 \subsubsection{Main results}
Recall that {\it the integral trace form} over $K$ is the integral quadratic form, denoted by $q_{K}$, that is obtained by restricting the rational trace form to the maximal order $O_{K}$. From an arithmetic point of view, the integral trace form is a better invariant than the rational trace form. One can see, for example, that the rational trace form does not even determine the discriminant of the number field: any $\Z/3\Z$-extension of $\Q$ has rational trace isometric to $\langle 1,1,1\rangle$(for a more general situation see \cite[Corollary I.6.5]{conner}). On the other hand, the integral trace can characterize the field in some non-trivial cases (see \cite{Manti}). It is natural then to wonder whether or not arithmetical equivalence implies equality between integral traces. An immediate observation that one can make from Perlis' work is that to ensure an isometry between the rational traces of two number fields, it is not necessary to have equality between their Dedekind zeta functions but only local information at finitely many places. With this observation in mind, we set course to find out if knowing the local root numbers is sufficient to determine the integral trace form. Explicitly we prove:

\begin{teo}[cf. Theorem \ref{samerootsametrace}]

Let $K,L$ be two non-totally real tamely ramified number fields of the same discriminant and signature. Then, the integral trace forms of $K$ and $L$ are isometric if and only if the $p$-local root numbers of $\rho_{K}$ and $\rho_{L}$ coincide for every odd prime $p$ that divides ${\rm disc}(K)$. 

\end{teo}

Since the Dedekind zeta function $\zeta_{K}$ determines  the discriminant and the signature of the field $K$,  see \cite{Perlis1}, Theorem \ref{samerootsametrace} gives a two fold generalization of \cite[Corollary 1]{Perlis}:
\begin{enumerate}
 
 \item On one hand, the conclusion of having isometric integral traces is stronger than having isometric rational traces.
 
 \begin{example}
 Let $K, L$ be two Galois cubic fields with different discriminant (take for instance the two cubic fields of discriminant $49$ and $81$ respectively). As pointed out before we have that $q_{K} \otimes \Q \cong q_{L} \otimes \Q \cong \langle 1,1,1\rangle,$ but clearly $q_{K} \not \cong q_{L}.$
 \end{example}
 
 \item On the other hand, the hypothesis of having the same local root numbers is weaker than that of having the same Dedekind zeta functions.
\begin{example}
Take any two non isomorphic tame Galois cubic fields of the same discriminant (take for instance the two cubic fields of discriminant $8281=7^2\cdot 13^2$). Since their integral traces are isomorphic (see \cite[Theorem 3.1]{Manti}) it follows from Theorem \ref{samerootsametrace} that they have the same root numbers at every prime. However, by Lemma \ref{solitario}, they do not have the same zeta function.
\end{example}
 
 \end{enumerate}
 We must however impose some necessary ramification restrictions so that the analogy is still valid in the integral case (see Remark \ref{necesario}).

\begin{teo}[cf Theorem \ref{AEimpliesSameTrace}]

Let $K,L$ be two non-totally real tamely ramified arithmetically equivalent number fields. Then, the integral trace forms $q_{K}$ and  $q_{L}$ are isometric.

\end{teo}

An interpretation of Theorem \ref{samerootsametrace} is that in order to know the integral trace, you only need local information from the Dedekind zeta function at  the ``bad"  places. Since the zeta function is a product of local $L$-functions, it is natural to wonder how those local factors, at the ramified places, influence the behavior of the integral trace. Inspired by this, we define the notion of {\it weak arithmetic equivalence} and show that indeed the integral trace is determined by local $L$-functions for a large family of number fields.

\begin{defi}[cf Definition \ref{LocArithEqui}]
Let $K,L$ be two number fields. We say that $K$ and $L$ are weakly arithmetically equivalent if and only if  

\begin{itemize}
\item A prime $p$ ramifies in K if and only if $p$ ramifies in $L$,
\item  $L_{p}(s,\rho_{K})=L_{p}(s,\rho_{L})$ for $p \in \{ p: p \mid {\rm disc}(K) \} \cup \{\infty\}.$   
\end{itemize}
\end{defi}

\begin{teo}[cf. Theorem \ref{LAEIsoTrace}]
Let $K,L$ be two weakly arithmetically equivalent number fields which are tame and non-totally real. Suppose that any of the following is satisfied:

\begin{itemize}

\item[(a)]  $K$ and $L$ have degree at most $3$ 

\item[(b)] $K$ has fundamental discriminant.\footnote{Recall that a discriminant is called {\it fundamental} if it is the discriminant of a quadratic field.}

\item[(c)] $K$ and $L$ are Galois over $\Q$.

\end{itemize}
\noindent 
Then, the integral trace forms of $K$ and $L$ are isometric. 

\end{teo}

\remark{See Question \ref{conjecture} and the remark after it for further thoughts on Theorem \ref{LAEIsoTrace}.}

\section{Local zeta functions and root numbers}\label{Dedekind}

\subsection{Background}
We start by recalling briefly how the Dedekind zeta function of a number field can be seen as the Artin $L$-function of a representation of the absolute Galois group $G_{\Q}$. See \cite{Mar} and \cite{Roh} for details and unexplained terminology.

\subsubsection{Dedekind zeta and Artin representations} 
Let $L$ be a number field with Galois closure $\tilde{L}$. Let $G(L):=\text{Gal}(\tilde{L}/\mathbb{Q})$ and  $H(L):=\text{Gal}(\tilde{L}/L)$. By composing the natural map $G_{\mathbb{Q}} \to G(L)$ with the natural action 
$G(L) \to \text{Sym}(G(L)/H(L))$, one gets a permutation representation $\pi_{L} \in H^{1}(G_{\mathbb{Q}}, S_{\text{deg}(L)})$ i.e.,  $\pi_{L} =\text{Inf}_{G(L)}^{G_{\mathbb{Q}}}(\text{Ind}_{H(L)}^{G(L)}1)$. Simply put, this is the natural action of $G_{\Q}$ on the embeddings of $L$ into $\tilde{L}$. The usual inclusions \[S_{\text{deg}(L)} \hookrightarrow \text{O}_{\text{deg}(L)}(\mathbb{C}) \hookrightarrow  \text{GL}_{\text{deg}(L)}(\mathbb{C}),\] together with $\pi_{L}$, yield an Artin representation $\rho_{L} \in H^{1}_{{\rm cont}}(G_{\mathbb{Q}}, \text{GL}_{\text{deg}(L)}(\mathbb{C})).$ By the induction property of Artin $L$-functions, the Dedekind zeta function $\zeta_{L}(s)$ of $L$ is nothing other than the Artin $L$-function $L(s,\rho_{L})$ associated to $\rho_{L}$. The function $L(s,\rho_{L})$ is defined as a product of local functions $L_{p}(s,\rho_{L})$ for each finite prime $p$, where the local parts are defined by restricting $\rho_{L}$ to a decomposition subgroup $G_{\mathbb{Q}_{p}}$.  By looking at the usual Euler product of $\zeta_{L}(s)$, we see that the local factors are given by \[L_{p}(s,\rho_{L})=\displaystyle \prod_{i=1}^{g}\left(\frac{1}{1-(p^{-s})^{f_{i}}}\right),\] where $g$ is the number of primes in $L$ lying over $p$ and the $f_{i}$'s are the residue degrees of a rational prime $p$ in its decomposition in $L$.\\ 

\paragraph{{\it Complete $L$-function and root numbers.}}
Given an Artin representation $\rho$, with Artin $L$-function $L(s,\rho)$, its complete $L$-function $\Lambda(s,\rho)$ is defined as \[\Lambda(s,\rho):=  A(\rho)^{s/2}L_{\infty}(s,\rho)L(s, \rho),\] where $A(\rho)$ is a positive integer divisible only by the finite primes at which $\rho$ ramifies, and $L_{\infty}(s,\rho)$ is a Gamma factor that depends on the value of $\rho$ at complex conjugation.  
The complete $L$-function satisfies a functional equation \[\Lambda(s,\rho)= W(\rho)\Lambda(1-s,\rho^{\vee}),\] where $\rho^{\vee}$ is the contragradient representation and $W(\rho)$ is a complex number called the {\it root number of} $\rho$. Due to a result of Deligne (see \cite{Deligne} and \cite[\S3]{Tate}) root numbers can be written as the product of the so called {\it local root numbers} \[ W(\rho)=\prod_{p}W_{p}(\rho).\]  The local root numbers $W_{p}(\rho)$ are complex numbers of norm $1$, moreover $W_{p}(\rho)=1$ whenever $\rho$ is unramified at $p$.

In the case of the permutation representation $\rho_{L}$, $A(\rho_{L})$ is equal to $|{\rm Disc}(L)|$. Moreover, since $\rho_{L}$ is an orthogonal representation it is a result of Fr\"ohlich and Queyrut, see \cite[\S3 Corollary 1]{Tate},  that $W(\rho_{L})=1$. The local infinite factor of $\rho_{L}$ is given by \[L_{\infty}(s,\rho_{L}):=\Gamma_{\mathbb{R}}^{r_1}(s)\Gamma_{\mathbb{C}}^{r_2}(s)\] where $r_{1}$ (resp $r_{2}$) is the number of real (resp complex) embeddings of $L$,  \[\Gamma_{\mathbb{R}}=(\pi)^{-s/2}\Gamma\left(\frac{s}{2}\right) \  {\rm and} \ \Gamma_{\mathbb{C}}=2(2\pi)^{-s}\Gamma(s),\] and  $\Gamma(s)$ is the usual Gamma function. We call the local root numbers $W_{p}(\rho_{L})$ the root numbers of the number field $L$.

\subsubsection{Root numbers and the trace}

The connection between root numbers $W_{p}(\rho_{L})$ of $\rho_{L}$ and the trace form $\tr_{L/\Q}(x^2)$ was first realized by Perlis by relating the results of Serre on Stiefel-Whitney invariants of the representation $\rho_{L}$ and those of Deligne on normalized root numbers.\\

\paragraph{{\it Second Stiefel-Whitney invariant and local root numbers.}} 
Let $L$ be a degree $n$ number field of discriminant $d$. The {\it second Stiefel-Whitney invariant} $w_{2}(L)$ of $L$, or of $\rho_{L}$, is a 2-torsion element in the Brauer group of $\Q$ defined as follows: Recall the standard presentation of the symmetric group $s_{n}$: \[ \left \langle t_{1},...,t_{n-1} : t_{i}^{2}=1 \mbox{ for $1 \leq i \leq n-1$},  (t_{i}t_{i+1})^{3}=1  \mbox{ for $1 \leq i \leq n-2$} , [t_{i},t_{j}]=1  \mbox{ for $ 2 \leq |i-j|$} \right \rangle \] Let $\tilde{s}_{n}$  be the $\pm 1$ central extension of $s_{n}$ defined by \[ \langle s_{1},...,s_{n-1},w : s_{i}^{2}=w^2=1 \mbox{ for $1 \leq i \leq n-1$},  (s_{i}s_{i+1})^{3}=1  \mbox{ for $1 \leq i \leq n-2$}, [s_{i},w]=1 \]  \[ \mbox{ for $ 1\leq i \leq n-1$}, [s_{i},s_{j}]=w  \mbox{ for $2 \leq |i-j|$}  \rangle \] where \[1 \longrightarrow \langle w \rangle \longrightarrow \tilde{s}_{n} \longrightarrow s_{n} \longrightarrow 1\] \[ \hspace{ .6 in} s_{i}  \longmapsto t_{i}.\] The extension $\tilde{s}_{n} \rightarrow s_{n}$ defines an element $\ell_{2} \in {\rm H}^{2}(s_{n}, \pm 1)$. By pulling back \[\pi_{L}: G_{\Q} \to s_{n} \quad {\rm to} \quad \pi^{*}_{L}:  {\rm H}^{2}(s_{n}, \pm 1) \to {\rm H}^{2}(G_{\Q}, \pm 1)\] one obtains the second Stiefel-Whitney invariant $w_{2}(L)=\pi^{*}_{L}(\ell_{2}).$ The local $p$-part $w_{2}(L)_{p}$ of $w_{2}(L)$ is the element of the Brauer group of ${\rm Br}_{2}(\Q_{p})$ obtained from $w_{2}(L)$  via restriction.

\begin{theorem}[Serre] Keeping the notation as the above, for all finite $p$ \[w_{2}(L)_{p}=h_{p}(q_{L})(2,d)_{p}\] where $h_{p}(q_{L})$ and $(\cdot, \cdot)_{p}$ denote the local $p$ Hasse-Witt invariant of the trace form  and the Hasse symbol, respectively.

\end{theorem}

\begin{proof}

See \cite[Th\'eor\`eme 1]{Serre}. 

\end{proof}

\begin{theorem}[Deligne] Keeping the notation as the above, for all finite $p$ \[w_{2}(L)_{p}=\frac{W_{p}(\rho_{L})}{W_{p}({\rm det} (\rho_{L}))}.\] 

\end{theorem}

\begin{proof}

See \cite[\S3 Theorem 3]{Tate}. 

\end{proof}
 
 An immediate consequence of Delingne's and Serre's is a formula relating the Hasse invariant of the trace form and the root numbers:
 
 \begin{corollary}\label{traceroot}
\[  h_{p}(q_{L})=(2,d)_{p}\frac{W_{p}(\rho_{L})}{W_{p}({\rm det}(\rho_{L}))}.\]
 
 \end{corollary}

\subsubsection{Background on the integral trace}

The following facts about the integral trace form will be useful in proving our main results. We included them here for the reader's convenience.

\paragraph{{\it The genus}}

The following Jordan decomposition of the local integral trace, for tame extensions, has been obtained by Erez, Morales and Perlis. For details, references and proofs see \cite{Manti3}.

\begin{theorem}\cite[Theorem 0.1]{Manti3}\label{GenusShape} Let $K$ be a degree $n$ number field and let $p$ be an odd prime which is at worst tamely ramified in $K$. Then, there exist $\alpha, \beta \in \Z_{p}^{*}$, and an integer $0<f \leq n$, such that \[q_{K} \otimes \Z_{p} \cong \underbrace{\langle 1,..., 1,\alpha \rangle}_{f}  \bigoplus \langle p \rangle \otimes \underbrace{\langle 1,...,1,\beta \rangle}_{n-f}.\]

\end{theorem}

\begin{corollary}\label{SameGenusTameSameRational}
Let $K,L$ be two tamely ramified number fields of the same discriminant and signature. Then, the integral trace forms $q_{K}$ and  $q_{L}$ are in the same genus if and only if $\displaystyle h_{p}(q_{K})=h_{p}(q_{L})$ for every odd prime $p$. 
\end{corollary}

\begin{proof} If $q_{K}$ and  $q_{L}$ are in the same genus then clearly they have the same local symbols at every prime. Conversely, let $p$ be an odd prime and suppose that $\displaystyle h_{p}(q_{K})=h_{p}(q_{L})$. Thanks to Theorem \ref{GenusShape} we can apply  \cite[Lemma 2.1]{Manti2} to the forms $q_{K} \otimes \Z_{p}, q_{L} \otimes \Z_{p}$ and conclude that \[q_{K} \otimes \Z_{p} \cong q_{L} \otimes \Z_{p}.\] Since $q_{K} \otimes \Z_{2} \cong q_{L} \otimes \Z_{2}$ (see \cite[Proposition 2.7]{Manti2}), and the fields have the same signature, the result follows.
\end{proof}

\paragraph{{\it The spinor genus}} For details, references and proofs about the spinor genus of the integral trace see \cite{Manti4}. 

\begin{theorem}\cite[Theorem 2.12]{Manti4}\label{SameSpinSameGen}
Let $K$ be a number field of degree at least $3$. Then, the genus of integral
trace form $q_{K}$ contains only one spinor genus.
\end{theorem}

The main application of the spinor genus is that it gives a way to determine when two number fields with ramification at infinity have isometric integral traces.

\begin{proposition}\label{ultimaprop}
Let $K, L$ be two non-totally real number fields. Then the forms $q_{K}$ and $q_{L}$ are in the same spinor genus if and only if they are isometric.
\end{proposition}

\begin{proof}
Since the discriminant and degree of a number field are invariants of the spinor (resp. isometry) class of its integral trace form, we may assume that both fields have degree $n \ge 3$. Since the fields are non-totally real, the forms $q_K$ and $q_L$ are indefinite and of dimension at least $3$.  By Eichler's Theorem (see \cite{eichler}) the spinor class and isometry class coincide for indefinite forms of dimension bigger than $2$, hence the result.
\end{proof}

\subsection{From root numbers to the integral trace}
 We now have all we need to give proofs to Theorems \ref{samerootsametrace} and \ref{AEimpliesSameTrace}.

\begin{lemma}\label{samerootdisc}
Let $K,L$ be number fields of the same discriminant. Then, for all primes $p$ \[W_{p}({\rm det} (\rho_{K}))=W_{p}({\rm det} (\rho_{L})).\]
\end{lemma}

\begin{proof}
Since $\rho_{K}$ is an orthogonal representation, the one dimensional representation \[\det (\rho_{K}) \to \pm 1\] factors through an injective morphism \[\delta: \textrm{Gal} (\Q({\sqrt{d}})/\Q) \to \pm 1,\] where $d \in \Q^*/(\Q^*)^{2}$ depends on $K$. Hence, if $\sigma \in G_{\Q}$, we have that $\det (\rho_{K})(\sigma)=\frac{\sigma(\sqrt{d})}{\sqrt{d}}.$ On the other hand, a calculation shows that $d ={\rm disc}(K)$ (see for example \cite[Pg 427, second paragraph]{Perlis}). Since $K$ and $L$ have the same discriminant, the representations $\det (\rho_{K})$ and $\det (\rho_{L})$ coincide hence so do their root numbers.
\end{proof}

\begin{theorem}\label{samerootsametrace}

Let $K,L$ be two tamely ramified number fields of the same discriminant and signature. Then, the integral trace forms $q_{K}$ and  $q_{L}$ are in the same spinor genus if and only if $\displaystyle W_{p}(\rho_{K})=W_{p}(\rho_{L})$ for every odd prime $p$ that divides ${\rm disc}(K)$. In particular, for a tame non-totally real number field, the integral trace form is completely determined by the local root numbers of the field.
\end{theorem}

\begin{proof}
We may assume that the fields have degree at least $3$. Thanks to Corollary \ref{traceroot} and Lemma \ref{samerootdisc} we have that for a prime $p$ \[ W_{p}(\rho_{K})=W_{p}(\rho_{L}) \ \mbox{if and only if} \ h_{p}(q_{K})=h_{p}(q_{L}).\]  On the other hand, it follows from Theorem \ref{SameSpinSameGen} and Corollary \ref{SameGenusTameSameRational} that the forms $q_{K}$ and  $q_{L}$ are in the same spinor genus if and only if $h_{p}(q_{K})=h_{p}(q_{L})$ for every odd prime $p$. Since  $h_{p}(q_{K})=h_{p}(q_{L})=1$ for unramified primes, the result follows. The last assertion in the theorem follows from Proposition \ref{ultimaprop} \end{proof}

As an immediate consequence of Theorem \ref{samerootsametrace}, we obtain a generalization of Perlis' result \cite[Corollary 1]{Perlis} to the integral trace.

\begin{theorem}\label{AEimpliesSameTrace}
Let $K,L$ be two non-totally real tamely ramified arithmetically equivalent number fields. Then, the integral trace forms $q_{K}$ and  $q_{L}$ are isometric.
\end{theorem}

\begin{proof}
Since arithmetically equivalent number fields share discriminants, signatures and local root numbers the result follows from Theorem \ref{samerootsametrace} and Proposition \ref{ultimaprop}.
\end{proof}

\begin{remark}\label{necesario}
In contrast to Perlis' result on the rational trace, arithmetic equivalence does not imply isometry between integral traces. In fact, as the following examples show, the ramification conditions imposed on the above theorem are not only sufficient but also necessary. Example 2.3 of \cite{Manti1} shows that the tameness condition in Theorem \ref{AEimpliesSameTrace} is necessary. On the other hand, if 
 $F$ and $L$  are the number fields defined by the polynomials
\begin{align*}
p_F&=x^7 - 2x^6 - 47x^5 + 25x^4 + 755x^3 + 496x^2 - 3782x - 5217,\\
p_L&=x^7 - 2x^6 - 47x^5 - 8x^4 + 480x^3 + 793x^2 + 233x + 19,
\end{align*}
it can be shown, as in the proof of \cite[Proposition 2.7]{Manti1}, that $F$ and $L$ are non-isomorphic arithmetically equivalent number fields. Furthermore, they are totally real and their common discriminant is equal to $5^2\cdot11^6\cdot19^4$ hence they are tamely ramified. A calculation in MAGMA shows that their integral traces are not equivalent. This example shows that the condition at infinity in Theorem \ref{AEimpliesSameTrace} is necessary.
\end{remark}

\subsection{Weak arithmetic equivalence}

After generalizing Perlis' work on arithmetic equivalence to the integral trace form, we are ready to go further by using {\it weak arithmetic equivalence}. To make statements about prime decomposition as general as possible, see for example \cite[Remark 2.6]{Manti2}, we use Conway's notation ($p=-1$) for the prime at infinity (\cite[Chapter 15, \S4]{conway}). 

\begin{definition}\label{LocArithEqui}
Let $K,L$ be two number fields. We say that $K$ and $L$ are weakly arithmetically equivalent if and only if 

\begin{itemize}

\item A prime $p$ ramifies in $K$ if and only if $p$ ramifies in $L$,
\item  $L_{p}(s,\rho_{K})=L_{p}(s,\rho_{L})$ for $p \in \{ p: p \mid {\rm disc}(K) \} \cup \{-1\}.$   
\end{itemize}
\end{definition}

\begin{remark} The second condition above should be interpreted as an equality between $L_{p}$-factors at every ramified prime. Of course there are fields in which $p=-1$ does not ramify, but in such cases $L_{-1}(s,\rho_{K})=L_{-1}(s,\rho_{L})$ is equivalent to $[K:\Q]=[L:\Q]$. Hence an equivalent statement to Definition \ref{LocArithEqui} is that $K$ and $L$ have same degree, same ramified primes and same local $p$-factors at such a primes.
\end{remark}

Recall that the {\it decomposition type} of a rational prime $p$ in a number field $K$ is the sequence $(f_{1},...,f_{g})$ consisting of the residue degrees $f_{i}$ of the primes in $K$ lying over $p$ written in increasing order: $f_{1}\leq ...\leq f_{g}.$

\begin{lemma}\label{RamtypeLAE}
Let $K,L$ be number fields and let $S_{K,L}$ be the set of primes $p$ that are ramified in either $K$ or $L$. Then, $K$ and $L$ are weakly arithmetically equivalent if and only if $K$ and $L$ have the same degree and for all $p \in S_{K,L} \cup \{-1\}$, we have that $p$ has the same decomposition type in $K$ and $L$.

\end{lemma}

\begin{proof}
This is a simple argument that can be found in the proof  vi) $\Rightarrow$ ii) of \cite[III, \S1 Theorem 1.3]{Klingen}. 
\end{proof}

We denote by $g_{p}^{K}$ the number of primes in $K$ lying above $p$. Additionally, we denote by $f_{p}^{K}$ the sum of the residue degrees of primes in $K$ above $p$.
 
\begin{corollary}\label{DetTameLAE}
Let $K$ and $L$ be weakly arithmetically equivalent number fields. Suppose that both fields are tame. Then, they have the same discriminant.

\end{corollary}

\begin{proof}
Thanks to Lemma \ref{RamtypeLAE} we know that $[L:\Q]=[K:\Q]$ and that $f_{p}^{K}=f_{p}^{L}$ for every prime $p$. Since both extensions are tame, we have by \cite[III, Proposition 13]{Serre2} that 
\[{\rm disc}(K)=\prod_{p}p^{[K:\Q]-f_{p}^{K}}=\prod_{p}p^{[L:\Q]-f_{p}^{L}}= {\rm disc}(L).\]

\end{proof}

Recall that a number field $L$ is called {\it arithmetically solitary} or {\it solitary} if one has that $K$ is isomorphic to $L$ for any number field $K$ arithmetically equivalent to $L$ .

\begin{remark}
The notion of weak arithmetic equivalence is quite less restrictive than that of arithmetic equivalence. For instance, there exist pairs of non isomorphic weakly arithmetically equivalent number fields 
which are either:

\begin{itemize}

\item[(1)] Galois extensions of $\Q$,

\item[(2)]  number fields with fundamental discriminant,

\item[(3)]  number fields of degree smaller than $7$.

\end{itemize}

\end{remark}

\begin{example}

The following polynomials, found by using \cite{Jones}, define pairs of non isomorphic weakly arithmetically equivalent number fields satisfying, respectively, conditions (1), (2) and (3) in the above remark.

\begin{itemize}

\item[(1)] The polynomials $x^7 - 609x^5 - 2233x^4 + 36743x^3 + 62118x^2 - 576520x + 3625$ and $x^7 - 609x^5 - 2233x^4 + 48111x^3 - 40194x^2 - 87696x + 77517$ define two non isomorphic Galois extensions, with Galois group $\Z/7\Z$, that are weakly arithmetically equivalent.

\item[(2)] The polynomials $x^6-14x^4-5x^3+52x^2+33x-24$ and $x^6-3x^5-17x^4-x^3+37x^2+27x+5$ define two non isomorphic weakly arithmetically equivalent number fields with fundamental discriminant equal to $725517561=3*241839187.$

\item[(3)] The polynomials $x^3 - 8x - 15$ and $x^3 + 10x - 1$ define two non isomorphic weakly arithmetically equivalent cubic fields.

\end{itemize}

\end{example}

In contrast, for arithmetic equivalence we have: 

\begin{lemma}\label{solitario}

Let $L$ be a number field satisfying either (1),(2) or (3) of the above remark. Then, $L$ is solitary.
\end{lemma}

\begin{proof}

Items (1) and (3) are part of \cite[III, \S1 Theorem 1.16]{Klingen}.  Item (2) follows from \cite[Theorem 1]{Kondo} and \cite[III, \S1 Theorem 1.16.c]{Klingen}.
\end{proof}

We now show that weak arithmetical equivalence determines the integral trace form for a large family of number fields.

\begin{theorem}\label{LAEIsoTrace}
Let $K,L$ be two weakly arithmetically equivalent number fields which are tame and non-totally real. Suppose that any of the following is satisfied:

\begin{itemize}

\item[(a)]  $K$ and $L$ have degree at most $3$

\item[(b)] $K$ has fundamental discriminant.

\item[(c)]  $K$ and $L$ are Galois over $\Q$.

\end{itemize}
\noindent 
Then, the integral trace forms $q_{K}$ and  $q_{L}$ are isometric.

\end{theorem}

\begin{proof}

Part (a) follows from Corollary \ref{DetTameLAE} and \cite[Theorem 3.3]{Manti2}. Thanks to Lemma \ref{RamtypeLAE}, we have that $g_{p}^{K}=g_{p}^{L}$ for all ramified prime $p$, hence (b) follows from \cite[Theorem 2.15]{Manti2}. If both fields are Galois, then not only every ramified prime $p$ has the same decomposition type in both fields, but it also has the same ramification index. This follows since both fields have the same degree and discriminant(see Corollary \ref{DetTameLAE}). Hence, part (c) follows from \cite[Proposition 2.14]{Manti2}.

\end{proof}

\begin{remark}
 Notice that under the restrictions imposed, Theorem \ref{LAEIsoTrace} gives a positve answer to the following natural question:
\end{remark}
 
 \begin{question}\label{conjecture}
 Let $K,L$ be two weakly arithmetically equivalent number fields which are tame and non-totally real. Are the integral trace forms $q_{K}$ and  $q_{L}$ isometric? 
 \end{question}

If we remove the signature condition in Theorem \ref{LAEIsoTrace}, we can't assure the existence of an isometry between the integral traces. However, by the same argument in the above proof, one sees that $q_{K}$ and  $q_{L}$ belong to the same spinor genus. In particular, thanks to Theorem \ref{samerootsametrace},  we have that $K$ and $L$ have the same local root numbers. Hence, we have:

\begin{theorem}\label{LAEIsoLocRoot}
Let $K,L$ be two tame weakly arithmetically equivalent number fields. Suppose that any of the following is satisfied:

\begin{itemize}

\item[(a)]  $K$ and $L$ have degree at most $3$.

\item[(b)] $K$ has fundamental discriminant.

\item[(c)]  $K$ and $L$ are Galois over $\Q$.

\end{itemize}

\noindent 
Then, $K$ and $L$ have the same local root numbers at every $p$.

\end{theorem}

Notice that Question \ref{conjecture} can be stated in terms of the local behavior of the Artin $L$-function $L(s,\rho_{L})$ and without any reference to the trace form. Explicitly, thanks to Theorem \ref{samerootsametrace}, Question \ref{conjecture} is equivalent to asking whether or not the equality at all ramified primes $p$ between $L_{p}(s,\rho_{K})=L_{p}(s,\rho_{L})$ implies equality between the local root numbers of $\rho_{K}$ and $\rho_{L}$ for any pair of  number fields $K,L$  that are tame and non-totally real. Since the signature condition we imposed on the number fields is only necessary to get isometry between the integral traces, and not only to get local isometry, we can omit that hypothesis and formulate \ref{conjecture} slightly in more general terms:

\begin{question}\label{conjecture2}
 Let $K,L$ be two tame weakly arithmetically equivalent number fields. Does it follow that $K$ and $L$ have the same local root numbers at every prime $p$?
 \end{question}

\paragraph{{\it Elliptic curves.}} 
A natural analog to the Dedekind zeta function $L(s, K)$ of a number field is the $L$-function $L(s,E)$ of a rational elliptic curve. Using the $\ell$-adic Tate module, for some prime $\ell$, one sees that $L(s,E)$ is the Artin $L$-function of  a $\Z_{\ell}$-representation of $G_{\Q}$. The notion of arithmetic equivalence in this context is equivalent to the one of isogeny class, thanks to Falting's isogeny Theorem. Since this equivalence is quite restrictive, it seems interesting to see what kind of invariants of an elliptic curve are determined by the analog notion of weak arithmetic equivalence. In particular, it is natural to ask if the analog to Question \ref{conjecture2} is valid in this context. It turns out that for semistable elliptic curves this is the case:

\begin{theorem}[Rohrlich]\label{rohrlic}
Let $E/\Q$, $E'/\Q$ be two semistable elliptic curves with bad ramification at the same primes. Suppose that for every bad prime $p$, the local  Hasse-Weil functions of $E$ and $E'$ coincide
\[L_{p}(s,E)=L_{p}(s,E').\] Then, for every prime $p$, $E$ and $E'$ have the same local root numbers \[W_{p}(E)=W_{p}(E').\]
\end{theorem}

\begin{proof}
This follows immediately from Rohrlich's formula for local root numbers \cite[Proposition 3]{Roh1}. 
\end{proof}

In our analogy between rational elliptic curves and number fields, the conductor plays the role of the discriminant. Henceforth, we can think of semistability for an elliptic curve as the analog, for a number field, of having square free discriminant. Keeping in mind this analogy we see that Theorem \ref{LAEIsoLocRoot} part (b) is the number theoretic version of Rohrlich's theorem. The following shows, as in the case of elliptic curves, that the hypothesis of having square free (conductor/discriminant) can not be removed from Theorem \ref{LAEIsoLocRoot}. In particular, the following gives a negative answer to Questions \ref{conjecture} and \ref{conjecture2}.

\begin{lemma}
Let $K$ and $L$ be the number fields defined by $x^4 - x^3 + 4x^2 + 68x + 152$ and $x^4 - 15x^2 - 21x + 121$ respectively. Then, $K$ and $L$ are tame non-totally real weakly arithmetically equivalent number fields with different root numbers at $p=7$ and $p=43$.
\end{lemma}

\begin{proof}
The fields $K$ and $L$ have signature $(0,2)$ and discriminant $d=(7\cdot13\cdot43)^{2}$. In particular, $K$ and $L$ are tame. Let $S=\{7,13,43\}$. The following table contains, for each prime $p$ in $S$, its decomposition type $(f_1,...,f_g)$, and respective ramification indices $(e_1,...,e_g) $, in the fields $K$ and $L$. \\
\[\begin{tabular}{|l|c|c|c|}\hline
\thead{$p$}&
\thead{7}&\thead{13} &\thead{43}\\    \hline
%\theadfont\diagbox[width=18em]{Field}{$(f_1,...,f_g) \ (e_1,...,e_g) $}      & &  &  \\    \hline
\thead{$K$} & $(1,1) \ (1,3)$ & $(1,1) \ (2,2)$ & $(1,1) \ (2,2)$ \\    \hline
\thead{$L$} & $(1,1) \ (2,2)$ & $(1,1) \ (1,3)$ & $(1,1) \ (1,3)$ \\    \hline

\end{tabular}\]
It follows from Lemma \ref{RamtypeLAE} that $K$ and $L$ are weakly arithmetically equivalent. Using the decomposition given in \cite[Theorem 0.1]{Manti3}, or by direct computation, we see that 
\[q_{K} \otimes \Z_{7} \cong \langle 1,3,7,21 \rangle, \ q_{L} \otimes \Z_{7} \cong \langle 1,1,7,7 \rangle\]
and  
\[q_{K} \otimes \Z_{43} \cong \langle 1,1,43,43 \rangle, \ q_{L} \otimes \Z_{43} \cong \langle 1,3,43,129 \rangle.\] In particular, \[h_{7}(q_{K})=1\neq-1=h_{7}(q_{L}) \ {\rm and} \ h_{43}(q_{K})=-1\neq 1= h_{43}(q_{L}).\] Therefore, arguing as in the first part of the proof of Theorem \ref{samerootsametrace}, we see that \[W_{p}(\rho_{K}) \neq W_{p}(\rho_{L})\] for $p =7,43$.
\end{proof}

\section*{Acknowledgements}
In the first place I would like to thank the referee for the various constructive and quite
valuable comments and suggestions on the paper. I thank  Lisa (Powers) Larsson for her helpful comments on a previous version of this paper, and to Val\'ery Mah\'e for providing me with a reference for Rohrlich's theorem.

\noindent
Guillermo Mantilla-Soler\\
Departamento de Matem\'aticas,\\
Universidad de los Andes, \\
Carrera 1 N. 18A - 10, Bogot\'a, \\
Colombia.\\
g.mantilla691@uniandes.edu.co

\end{document}